\documentclass[12pt]{amsart}

\usepackage{amssymb, amsmath}
\usepackage{amsthm}
\usepackage{mathrsfs}
\usepackage[pdfborder={0 0 0 [0 0 ]},bookmarksdepth=3, bookmarksopen=true,unicode=true]{hyperref}
\usepackage{enumitem}
\usepackage{cases}
\usepackage[font={it}, margin=1cm]{caption}
\usepackage{verbatim}
\usepackage[capitalize]{cleveref}

\usepackage{color}
\usepackage[ps,all,arc,rotate]{xy}

\usepackage{tikz}
\usetikzlibrary{decorations.markings, decorations.pathreplacing, positioning,arrows, matrix, decorations.pathmorphing, shapes.geometric, calc}
\usetikzlibrary{backgrounds, decorations, cd}

\newtheorem{theorem}{Theorem}[section]
\newtheorem{question}[theorem]{Question}
\newtheorem{lemma}[theorem]{Lemma} 
\newtheorem{proposition}[theorem]{Proposition} 
\newtheorem{conjecture}[theorem]{Conjecture} 
\crefname{conjecture}{Conjecture}{Conjectures}
\newtheorem{corollary}[theorem]{Corollary} 

\theoremstyle{definition}
\newtheorem{dfn}[theorem]{Definition}

\theoremstyle{remark}

\newtheorem{remark}[theorem]{Remark}

\newcommand{\BS}{\operatorname{BS}}

\newcommand{\Z}{\mathbb{Z}}

\newcounter{dawidcomments}

\newcounter{alancomments}

\newcounter{gilescomments}

\title[The Surface Group Conjectures with two generators]
{The Surface Group Conjectures for groups with two generators}
\author{Giles Gardam}
\address{
Mathematisches Institut, Universit\"at M\"unster, Einsteinstr.~62, 48149 M\"unster,
Germany}
\email{ggardam@uni-muenster.de}
\author{Dawid Kielak}
\address{
University of Oxford, Oxford, OX2 6GG,
UK}
\email{kielak@maths.ox.ac.uk}
\author{Alan D. Logan}
\address{
University of St Andrews, St Andrews, KY16 9SS,
UK}
\email{adl9@st-andrews.ac.uk}

\begin{document}

\begin{abstract}
The Surface Group Conjectures are statements about recognising surface groups among one-relator groups, using either the structure of their finite-index subgroups, or all subgroups.	
We resolve these conjectures in the two generator case. More generally, we prove that every two-generator one-relator group with every infinite-index subgroup free is itself either free or a surface group. 
\end{abstract}
\maketitle

\section{Introduction}
\label{sec:introduction}

The role of surface groups in the pro-$p$ world is played by Demushkin groups, introduced in 1961 in the study of Galois theory. In 1973, Ando\v{z}ski\u{\i} \cite{Andozhskii1973} characterised Demushkin groups as precisely those finitely generated non-cyclic one-related pro-$p$ groups of cohomological dimension 2, all of whose maximal subgroups are also one-relator pro-$p$ groups.
Motivated by this result, in 1980 Mel'nikov \cite[Problem 7.36]{KhukhroMazurow2014} asked: let $G$ be a residually finite one-relator group with every subgroup of finite index also a one-relator group; is $G$ either free or a surface group? (The question appears also in \cite[Question 2.16]{Baumslag2017survey}.)

Let us introduce the following.
\begin{dfn}[Mel'nikov group]
A \emph{Mel'nikov group} is a non-free infinite one-relator group with every subgroup of finite index also a one-relator group.
\end{dfn}
Compared to Mel'nikov's original question, we added non-freeness and being infinite to the definition (finite cyclic groups are obvious counterexamples to the question), but we removed
residual finiteness. The reason for this is that it is not clear what role this last assumption plays. It seems that it was included in Mel'nikov's question to strengthen the analogy with Ando\v{z}ski\u{\i}'s result.
	

Residually finite Mel'nikov groups are not necessarily surface groups, with the Baumslag--Solitar groups $\operatorname{BS}(1, n)$ forming a family of counter-examples; the Kourovka Notebook records that this was pointed out by \v{C}urkin in 1982. 
This observation prompted a number of results and conjectures trying to understand which algebraic properties force a one-relator group to be a surface group; we refer the reader to the survey of Baumslag--Fine--Rosenberger for more details \cite[Section 2.4]{Baumslag2017survey}, but we remark that the strongest result \cite[Corollary 5]{Wilton2012}, due to Wilton, is for limit groups (that is, finitely generated fully residually free groups).
The two remaining conjectures here are the Surface Group Conjectures A \& B, and are as follows.
Recall that a surface group is the fundamental group of a closed surface of non-positive Euler characteristic (i.e.\ a closed surface other than the sphere or projective plane) and note that this includes the torus and Klein bottle, giving the groups $\operatorname{BS}(1, 1)$ and $\operatorname{BS}(1, -1)$ respectively.

Firstly, Ciobanu--Fine--Rosenberger conjectured that the Baumslag--Solitar groups $\BS(1, n)$ are the only non-surface Mel'nikov groups, and therefore suggested the following conjecture \cite{Ciobanu2013surface}, see also \cite[Conjecture 2.17]{Baumslag2017survey}.

\begin{conjecture}[Surface Group Conjecture A]
\label{conj:MandRF}
Let $G$ be a residually finite Mel'nikov group. Then $G$ is a surface group or $\BS(1, n)$ for some non-zero integer $n$.
\end{conjecture}

The second conjecture was suggested by Fine \cite[Question OR15]{Baumslag2002Open}, \cite[Conjecture 1.2]{Fine2007surface}, \cite[Question 2.18]{Baumslag2017survey}.
It is weaker than \cref{conj:MandRF} since the groups it concerns are residually finite anyway, being finitely generated and free-by-cyclic \cite{Baumslag1971RF}.

\begin{conjecture}[Surface Group Conjecture B]
\label{conj:MandIF}
Let $G$ be a Mel'nikov group with every subgroup of infinite index free. Then $G$ is a surface group.
\end{conjecture}

We suggest a third conjecture of our own that is much stronger than \cref{conj:MandIF}.
We later justify this conjecture by proving that it follows from Gromov's famous Surface Subgroup Conjecture for Hyperbolic Groups (see \cref{thm:IF}).

\begin{conjecture}
\label{conj:IF}
Let $G$ be an infinite non-free one-relator group with every subgroup of infinite index free. Then $G$ is a surface group.
\end{conjecture}

Strebel proved that if $G$ is a Poincaré duality group, then all its infinite index subgroups have strictly smaller cohomological dimension \cite{Strebel1977}.
\Cref{conj:IF} addresses the question of when the converse holds for dimension $2$.

It is a standard fact that a one-relator group is two-generated if and only if it is (1) finite cyclic, so $G=\langle x\mid x^n\rangle$, or (2) a two-generator one-relator group (possibly $\Z$), so $G=\langle x,y \mid R\rangle$ with $R \neq 1$, or (3) the free group of rank 2, so $G=\langle x,y,z \mid R\rangle$ with $R$ primitive \cite[II.5.11]{LyndonSchupp2001} (this also can be deduced from a theorem of Stallings \cite[p.~171]{Stallings1965}).
Note that we adopt the convention that one-relator presentations have non-trivial relator; this is convenient when considering Euler characteristic, which indeed distinguishes between the 3 possibilities just given.
Two-generator one-relator groups are of central importance in the theory of one-relator groups.
In particular, all known pathological one-relator groups are two-generated, and Louder and Wilton have given a conceptual explanation for this importance \cite[Corollary 1.10 \& Conjecture 1.12]{Louder2018negative}.
Therefore, if the above conjectures were false, one might expect them to be false already when $G$ is two-generated.
This is however not the case.

\begin{theorem}
	\label{thm:main}
	The \cref{conj:IF,conj:MandIF,conj:MandRF} all hold when the group in question is two-generated.
\end{theorem}

The proof of \cref{conj:IF} (and thus \cref{conj:MandIF}) in the two-generator case follows from a more general theorem about indicable groups with vanishing first $L^2$-Betti number that have all infinite-index subgroups free. We show that such groups are necessarily either cyclic or $\BS(1,\pm1)$, see \cref{vanishing first l2}.
Note that we certainly need some assumptions on the group beyond having all infinite-index subgroups free, as for instance there exist non-cyclic groups all of whose proper non-trivial subgroups are infinite cyclic \cite{Olshanskii79noetherian}. 

\cref{vanishing first l2} relies on a technical statement, namely \cref{Ifreeprods}, that deals with torsion-free groups with vanishing first $L^2$-Betti number that satisfy the Atiyah conjecture, and in which all normal infinite-index subgroups are free products of finitely generated groups. We show that every epimorphism to $\mathbb Z$ from such a group has finitely generated kernel. Another way to say this is that such groups must have the first BNS-invariant equal to the entire character sphere.


We frequently use a number of standard facts about one-relator groups: that they contain torsion if and only if the relator is a proper power \cite[Theorem 4.13]{mks}; that in the torsion-free case the presentation complex is aspherical \cite{Lyndon1950, Cockcroft1954}, and hence that the Euler characteristic $\chi(G)$ of $G=\langle \mathbf x \mid R\rangle$ is $2-|\mathbf x|$; that in the torsion-free case the first $L^2$-Betti number is equal to the Euler characteristic \cite{DicksLinnell2007}.

\begin{remark}
After the first version of this article appeared online, Jack Button informed the authors of an alternative proof of \cref{conj:IF} in the two generator case, using the Coherence Theorem of Feighn--Handel~\cite{FeighnHandel1999} in place of $L^2$-Betti numbers to show finite generation of $K$ in the proof of \cref{vanishing first l2}.
\end{remark}

\subsection*{Acknowledgements}
The authors are grateful to Ilir Snopce for pointing out that Mel'nikov was motivated by results coming from the theory of pro-$p$ groups, and to Henry Wilton for insightful comments.

This work has received funding from
the European Research Council (ERC) under the European Union's Horizon 2020 research and innovation programme (Grant agreement No. 850930),
and from the Engineering and Physical Sciences Research Council (EPSRC), grants EP/R035814/1 and EP/S010963/1,
and was supported by the Deutsche Forschungsgemeinschaft (DFG, German Research Foundation) -- Project-ID 427320536 -- SFB 1442, as well as under Germany's Excellence Strategy EXC 2044--390685587, Mathematics M\"unster: Dynamics--Geometry--Structure.

\section{Groups with two generators}

In this section we prove \cref{thm:main}.

\subsection{Infinite-index subgroups}

First we will focus on \cref{conj:IF}, on groups where all subgroups of infinite index are free.

\begin{lemma}
	\label{torsion IF}
	Let $G$ be an infinite group. If all infinite-index subgroups of $G$ are free, then $G$ is torsion free.	
\end{lemma}	
\begin{proof}
	Let $g \in G$ be of finite order. The subgroup it generates is finite, and hence of infinite index in $G$, and hence free. But it is also finite, and therefore must be trivial. Thus $g=1$.
\end{proof}

Since in \cref{conj:IF} we are interested in infinite groups, we see now that in fact we are interested in torsion-free one-relator groups. When $|\mathbf x |=2$, such groups have Euler characteristic equal to zero, and so have first $L^2$-Betti number equal to zero \cite{DicksLinnell2007}. Hence, in the next couple of results we will focus on groups with first $L^2$-Betti number equal to zero, bearing in mind that groups of interest here will satisfy this property. The results that now follow are however more general.

Recall that an \emph{agrarian map} is a unit-preserving ring homomorphism $\alpha \colon \Z G \to \mathbb K$ where $\Z G$ is the integral group ring of a group $G$, and where $\mathbb K$ is a skew-field. This notion was introduced in \cite{Kielak2020}. There are three agrarian maps that are frequently used: one arises from the Atiyah conjecture, and the other two are used in defining Alexander polynomials, the multivariate and univariate kinds. All three share some structural properties, so we will stay in the level of generality that may then be applied to all three situations.

Let $\rho \colon G \to \mathbb Z$ be a quotient map and $t \in G$ be such that $\rho(t)$ generates $\Z$. Suppose that we have a skew-field $\mathbb L$ and a unit-preserving ring morphism $\alpha \colon \Z \ker \rho \to \mathbb L$. Suppose further that the action of $t$ on $\ker \rho$ by conjugation descends to an action on $\alpha(\Z \ker \rho)$, and that it can be extended to an action on the whole of $\mathbb L$. Using this extension, we form 
the ring of twisted Laurent polynomials $\mathbb L[t,t^{-1}]$; we may now also extend the domain of $\alpha$ and obtain a ring morphism
\[
\alpha \colon \Z G \to \mathbb L[t,t^{-1}].
\]
Finally, we let 
$\mathbb K$ 
be the Ore localisation of $\mathbb L[t,t^{-1}]$, and we follow $\alpha$ by the inclusion $\mathbb L[t,t^{-1}] \to \mathbb K$.
Let us now see how this general setup incorporates three situations alluded to above.

First, we may start with $\alpha\vert_{\Z \ker \rho}$ being the augmentation map followed by embedding $\Z$ into $\mathbb Q$. We then get $\mathbb L = \mathbb Q$ with trivial $t$-action. This choice is used in computing the univariate Alexander polynomial corresponding to the map $\rho$.

Second, we may take $\alpha\vert_{\Z \ker \rho}$ to be the ring map induced by the group homomorphism taking $\ker \rho$ first to $G$ and then to $H_1(G;\mathbb Q)$. Denoting the $\mathbb Q$-generators of the image of $\ker \rho$ in $H_1(G;\mathbb Q)$ by $s_1, \dots, s_n$, we obtain a ring map 
\[\alpha\vert_{\Z \ker \rho} \colon \Z \ker \rho \to \mathbb Z[s_1, {s_1}^{-1}, \dots, s_n, {s_n}^{-1} ].\]
We then change scalars from $\Z$ to $\mathbb Q$, and Ore-localise to the field of rational functions $\mathbb L =  \mathbb Q(s_1, \dots, s_n)$. The action of $t$ on $\mathbb L$ is trivial. 
This choice is used in the definition of the multivariate Alexander polynomial.

Third, when $G$ is torsion free and satisfies the Atiyah conjecture, we may take $\mathbb K$ to be the Linnell skew-field $\mathcal D(G)$ of $G$, and $\mathbb L$ to be the Linnell skew-field $\mathcal D(\ker \rho)$ of $\ker \rho$. The theory of Linnell skew-fields guarantees that $\mathbb K$ is the Ore localisation of $\mathbb L[t,t^{-1}]$, see for example \cite[Subsection 4.1]{Kielak2020}. The action of $t$ comes from the conjugation action within $\mathcal D(G)$. The map $\alpha \colon \Z G \to \mathcal D(G)$ is as described above, and its existence comes straight from the theory of Linnell skew-fields. This is the situation that is key for applications in this article.

We will proceed now with our general setup. Note that as $\mathbb K$ is a skew-field and a $G$-module, we may define the Betti numbers
\[
\beta_i^\mathbb K(G) = \dim_\mathbb K H_i(G;\mathbb K).
\]
When $\mathbb K = \mathcal D(G)$, these Betti numbers are precisely the $L^2$-Betti numbers of $G$.

\begin{proposition}
	\label{first L2 Betti}
	Let $G=K\rtimes \mathbb Z$ be finitely generated, and let $\mathbb K$ be as described above.
	If $\beta_1^{\mathbb K}(G) = 0$, then $\beta_1^{\mathbb K}(K)$ is a non-negative integer. 
\end{proposition}
\begin{proof}
	
	We let $\rho \colon G \to \mathbb Z$ denote the projection of $K\rtimes \mathbb Z$ onto the second factor. We pick a finite generating set $\mathbf x$ of $G$ with an element $t \in \mathbf x$ such that $\rho(t)$ generates $\mathbb Z$, and $\rho(x)$ is trivial for every $x \in \mathbf x \smallsetminus \{t\}$. 
	This means that to compute the first group homology of $G$ we may look at the chain complex of right $\mathbb Z G$-modules $C_\bullet = (C_i,\partial_i)$ with $C_i =0$ for all $i>2$ and with the first three terms as follows:
	\[
	\bigoplus_{\mathbf R} \mathbb Z G \xrightarrow{\partial_2} \bigoplus_{\mathbf x} \mathbb Z G \xrightarrow{\partial_1} \mathbb Z G
	\]
	where $\mathbf R$ is some possibly infinite indexing set (corresponding to relations in $G$). The modules above come equipped with obvious bases as right $\mathbb Z G$-modules. Clearly, the elements of these bases acted on by integer powers of $t$ form right $\mathbb Z K$-bases. Also, $\partial_1$ may be identified with the $(1 \times \mathbf x)$-matrix over $\mathbb Z G$ with entries $1-x$ for $x \in \mathbf x$.
	
	The same chain complex, now thought of as a chain complex of free $\mathbb Z K$-modules, computes the first group homology of $K$. More explicitly, 
	\[
	\beta_1^{\mathbb K}(K) = \dim_{\mathbb K}H_1(C_\bullet \otimes_{\mathbb Z K} \mathbb K).
	\]
	Since $\alpha(K) \subseteq \mathbb L$, and since $\mathbb L$ and $\mathbb K$ are both skew-fields, we see that
	\[
	\beta_1^{\mathbb K}(K) = \beta_1^{\mathbb L}(K).
	\]
	
	Suppose that there exists $s \in \mathbf x \smallsetminus \{t\}$ such that the image of $1-s \in \mathbb Z K$ in $\mathbb K$ is non-trivial; note that the image automatically lies in $\mathbb L$.
	We change the basis of the module $C_1 \otimes_{\mathbb Z G} \mathbb L [t^{-1},t] = \bigoplus_{\mathbf x} \mathbb L [t^{-1},t]$ by multiplying it on the left by the matrix equal to the identity except in the $s$-column, in which the $s$-entry is $\alpha(1-s)^{-1}$, and the $s'$-entry is \[-\alpha(1-s')\alpha(1-s)^{-1}\]
	 for every $s' \in \mathbf x \smallsetminus \{s\}$. The matrix is invertible over $\mathbb L [t^{-1},t]$ since it is a product of elementary matrices and a diagonal matrix with non-zero diagonal entries. After this change of basis, the differential $\partial_1\otimes_{\Z G} \mathrm{id}$ 
	becomes the projection onto the factor of $\bigoplus_{\mathbf x} \mathbb L [t^{-1},t]$ corresponding to $s$. In particular, $\ker \partial_1 \otimes_{\mathbb Z G} \mathrm{id}$ coincides with $\bigoplus_{\mathbf x \smallsetminus \{s\}} \mathbb L [t^{-1},t]$. 
	Since the change of basis is an automorphism of the $\mathbb L [t^{-1},t]$-module $C_1 \otimes_{\mathbb Z G} \mathbb L [t^{-1},t]$, it is automatically an isomorphism, and hence a change of basis, of the underlying $\mathbb L$-module, which is isomorphic to $C_1 \otimes_{\mathbb Z K} \mathbb L$.
	This means that we may indeed change the basis of $C_1 \otimes_{\mathbb Z G} \mathbb L[t,t^{-1}]$ as above without changing $\beta_1^{\mathbb K}(G)$ and $\beta_1^{\mathbb L}(K)$.
	In what follows we will work in this new basis.
	
	
	If no such $s$ exists, then the matrix representing $\partial_1$ after passing to $\mathbb L [t,t^{-1}]$ has all entries zero, except one, which is equal to $1-t$. Hence $\ker \partial_1 \otimes_{\mathbb Z G} \mathrm{id}$ coincides with $\bigoplus_{\mathbf x \smallsetminus \{t\}} \mathbb L [t^{-1},t]$. In this case we set $s = t$ for notational convenience.
	
	Now $\ker \partial_1 \otimes_{\mathbb Z G} \mathrm{id}$ in $C_1 \otimes_{\mathbb Z G} \mathbb K$ coincides with $\bigoplus_{\mathbf x \smallsetminus \{s\}} \mathbb K$. Since $H_1(C_\bullet \otimes_{\mathbb Z G} \mathbb K) = 0$, every element of $\bigoplus_{\mathbf x \smallsetminus \{s\}} \mathbb K$ lies in the image of $\partial_2 \otimes_{\mathbb Z G} \mathrm{id}$. Since $\mathbb K$ is the Ore localisation of $\mathbb L [t^{-1},t]$, this implies that every factor of $\bigoplus_{\mathbf x \smallsetminus \{s\}} \mathbb L [t^{-1},t]$ contains a non-zero element, say $p_x$, $x\in \mathbf x \smallsetminus \{s\}$, lying in the image of 
	\[\partial_2 \otimes_{\mathbb Z G} \mathrm{id} \colon C_2 \otimes_{\mathbb Z G} \mathbb L[t,t^{-1}] \to C_1 \otimes_{\mathbb Z G} \mathbb L[t,t^{-1}].\]
	 Since $p_x$ is Laurent polynomial in $t$, it has a \emph{degree}, namely the difference between the highest and the lowest power of $t$ appearing in $p_x$. It is immediate that 
	\[\beta_1^{\mathbb K}(K) = \beta_1^{\mathbb L}(K) = \dim_{\mathbb L} H_1(C_\bullet \otimes_{\mathbb Z K} \mathbb L) \]
	 is bounded above by the sum of the degrees of the polynomials $p_x$, which is a non-negative integer.
\end{proof}

Now we are going to apply the above considerations to groups in which we have some control over infinite-index subgroups.

\begin{proposition}
	\label{Ifreeprods}
	Let $G$ be a torsion-free group satisfying the Atiyah conjecture. Let $\phi \colon G \to \mathbb Z$ be an epimorphism whose kernel $K$ is a free product of finitely generated groups. If the first $L^2$-Betti number of $G$ is equal to zero, then $K$ is finitely generated.
\end{proposition}	
\begin{proof}
	Since $G$ is torsion-free, $K$ is a free product of infinite groups. The first $L^2$-Betti number of such a group is bounded below by the number of non-trivial free factors minus 1 by L\"uck's appendix to \cite{Brownetal2008}, and hence \cref{first L2 Betti} tells us that $K$ can only have finitely many free factors. It follows that $K$ is itself finitely generated.
\end{proof}

\begin{theorem}
	\label{vanishing first l2}
	Let $G$ be a finitely generated group with first $L^2$-Betti number equal to zero, such that $G$ admits an epimorphism to $\mathbb Z$. If all infinite-index subgroups of $G$ are free, then $G$ is isomorphic to either $\mathbb Z$ or $\mathrm{BS}(1,\pm1)$. 
\end{theorem}

\begin{proof}
	The assumptions give us an epimorphism $G \to \mathbb Z$ with kernel $K$ that is necessarily free. It follows that $G$ is torsion free and that it satisfies the Atiyah conjecture (by the work of Linnell~\cite{Linnell1993}).
 	
	By \cref{Ifreeprods}, 
	$K =F_n$ for some $n$, and the action of $\mathbb Z$ is given by an automorphism $\phi \in \mathrm{Aut}(F_n)$. 
	If $n=0$ then $G=\mathbb Z$. Let us now assume that $G \neq \mathbb Z$, and hence that $n\geqslant 1$.
	
	If $G$ is not hyperbolic then $\phi$ is not atoroidal, as proven by Brink\-mann~\cite{Brinkmann2000}, and so $G$ contains $\mathbb Z^2$. This last subgroup must be of finite index, and therefore $G$ cannot contain a non-abelian free group. We conclude that $n=1$ and the result follows.
	
	If $G$ is hyperbolic then it is virtually RFRS by the work of Hagen--Wise~\cite{HagenWise2015} and Agol~\cite{Agol2013Haken}. Every non-abelian finitely generated RFRS group $G$ admits a finite-index subgroup $H$ whose abelianisation has rank greater than the original group. If our group $G$ is virtually abelian, then being hyperbolic and torsion-free it is isomorphic to $\mathbb Z$, which we have already discounted.
    Thus $G$ contains a finite-index subgroup $H$ with abelianization of rank at least $2$.
    The subgroup $H$ is moreover hyperbolic and all of its infinite-index subgroups are free.
	Now, again by \cref{Ifreeprods}, all epimorphisms $H \to \mathbb Z$ have finitely generated kernels. Thus, the first BNS invariant $\Sigma^1(H)$ coincides with the entire character sphere of $H$, and therefore $H$ fits into a short exact sequence
	\[
	N \to H \to \mathbb Z^m
	\]
	with $N$ finitely generated, $m\geqslant 2$, and with $\mathbb Z^m$ being the free part of the abelianisation of $H$, see \cite[Theorem B1]{Bierietal1987}.
	The free group $F_n$ in $H$ is the kernel of an epimorphism $H \to \mathbb Z$, and every such epimorphism has to factor through $\mathbb Z^m$ above. Hence $F_n$ itself fits into a short exact sequence
	 \[
	 N \to F_n \to \mathbb Z^{m-1}
	 \]
	with $m-1 > 0$. This implies that $N$ is a normal, finitely generated subgroup of $F_n$ of infinite index. The only such subgroup is the trivial group, which forces $n=1$ and $m=2$.
 So $H$ must be isomorphic to $\mathbb Z^2$, which is not hyperbolic. This is a contradiction.
\end{proof}

\begin{corollary}
	\label{corol:2GenIF}
	If all infinite-index subgroups of $G=\langle x, y\mid R\rangle$ are free, then 
	either $G$ is free or $G \cong \mathrm{BS}(1,\pm1)$.
\end{corollary}
\begin{proof}
	By \cref{torsion IF}, $G$ is torsion free. Therefore either $R$ is trivial, and so $G$ is free, or $R$ is not trivial, and hence the first $L^2$-Betti number of $G$ is zero. In the latter case we may use \cref{vanishing first l2}.
\end{proof}

Clearly the above corollary confirms \cref{conj:IF,conj:MandIF} when $|\mathbf x|=2$.

\subsection{Mel'nikov groups}
Next we turn to \cref{conj:MandRF} on Mel'nikov groups, i.e.\ on one-relator groups all of whose finite index subgroups are themselves one-relator.
First we make a general observation.

\begin{proposition}
	\label{torsion}
	Every Mel'nikov group is torsion-free.
\end{proposition}

\begin{proof}
	A one-relator group $G=\langle \mathbf{x}\mid R\rangle$ has torsion if and only if the word $R\in F(\mathbf{x})$ is a proper power, that is, $R=S^n$ for some $n>1$ maximal. Such a group $G$ has Euler characteristic $\chi(G)=1-|\mathbf{x}|+1/n$ \cite[Theorem~4]{Chiswell1976Euler}, which in particular is not an integer.
	
	So, let $G=\langle \mathbf{x}\mid S^n\rangle$ be an infinite one-relator group with torsion. Then $G$ surjects onto $\mathbb{Z}$, and every torsion element of $G$ is contained in the kernel of this map. Hence, $G$ surjects onto the cyclic group of order $n$ with kernel $K$ containing all the torsion elements of $G$. Therefore, $K$ is not torsion-free. Also,
	\[\chi(K)=n\cdot\chi(G)=n-n|\mathbf{x}|+1\in\mathbb{Z}\textrm{.}\]
	Therefore, $K$ contains torsion but has integer Euler characteristic, and so $K$ is not a one-relator group, again by \cite[Theorem~4]{Chiswell1976Euler}.
\end{proof}

Now we will use a theorem of Mann--Segal to resolve \cref{conj:MandRF} in the two-generated case; note that the theorem builds on the work of Lubotzky--Mann \cite{LubotzkyMann1989rank}, which in particular applies the Feit--Thompson Odd Order Theorem.

\begin{theorem}
	\label{prop:2GenMel}
	A two-generated residually finite Mel'nikov group is isomorphic to $BS(1, n)$ for some $n$.
\end{theorem}

\begin{proof}
	Let $G$ be a 2-generator Mel'nikov group.
	As we have just seen, $G$ is torsion free and its Euler characteristic is zero. This implies that the same two properties are enjoyed by all finite index subgroups of $G$, and hence each of these subgroups is specifically a \emph{two-generator} one-relator group.
	
	By a theorem of Mann and Segal \cite[Theorem A]{MannSegal1990uniform}, any residually finite group whose finite index subgroups have uniformly bounded rank is virtually solvable, so in particular $G$ has no non-abelian free subgroups.	
	
	Magnus rewriting allows us to write $G$ as an HNN extension with finitely generated vertex and edge groups, and with the edge group free. Since $G$ is not $\mathbb Z$, the edge group is non-trivial. We immediately conclude that the edge group is $\mathbb Z$. Also, if the attaching maps embedding the edge group into the vertex group both have proper image, then one easily constructs a non-abelian free subgroup in $G$ using Bass--Serre theory.
	Hence $G$ must be an ascending HNN extension of $\Z$, that is, it must be isomorphic to $\operatorname{BS}(1, n)$.
\end{proof}

\section{Connections}

In this section we connect \cref{conj:MandRF,conj:MandIF,conj:IF} to recent advances and other conjectures in Geometric Group Theory.

\subsection{Negative immersions}
\label{sec:immersions}

Louder--Wilton in \cite{Louder2018negative} introduced the concept of \emph{having negative immersions} for two-dimensional CW-complexes (for a comparison with the stronger notation of negative immersions in the sense of Wise \cite{Wise2020coherence}, see \cite[Section 3.4]{Louder2021uniform}).
One-relator groups with negative immersions are hyperbolic and virtually special \cite{linton2022one}.
Louder--Wilton give a group theoretic characterisation of negative immersions for presentation complexes of one-relator groups, and we use the following form of that characterisation.

\begin{theorem}[{\cite[Theorem 1.2]{Louder2021uniform}}]
	\label{thm:free_subgroups_detect_negative_immersions}
	The presentation complex of a one-relator group has negative immersions if and only if every two-generated subgroup is free.
\end{theorem}

The groups $G=\langle \mathbf{x}\mid R\rangle$ of \cref{conj:IF} have negative immersions when $|\mathbf{x}|\geqslant3$:

\begin{theorem}
	\label{thm:SGC_negative_immersions}
	Let $G=\langle \mathbf{x}\mid R\rangle$ be an infinite one-relator group with every subgroup of infinite index free. If $|\mathbf{x}|\geqslant 3$ then $G$ has negative immersions.
\end{theorem}

\begin{proof}
	The group $G$ is torsion-free by \cref{torsion IF}, and so $\chi(G) = 2 - |\mathbf{x}| \leqslant -1$.
	Any proper finite-index subgroup $H$ has rank at least 3: $\chi(H) \leqslant -2$ but $H$ is also 2-dimensional, so 
	\[\chi(H) = \beta_0(H) - \beta_1(H) + \beta_2(H) \geqslant 1 - \beta_1(H),\]
	 and thus $\beta_1(H) \geqslant 3$.
	This implies that all proper subgroups of $G$ that are two-generated are of infinite index, and thus free.

	We still need to check what happens if $G$ itself is two-generated.
    By a standard one-relator group fact, since $\chi(G) \leqslant -1$ this is only possible if $G \cong F_2$ is free.
	We are therefore done by Theorem~\ref{thm:free_subgroups_detect_negative_immersions}.
\end{proof}

However, we have been unable to prove the analogous result for the groups $G=\langle \mathbf{x}\mid R\rangle$ of \cref{conj:MandRF} when $|\mathbf{x}|\geqslant3$, so we have the following question.

\begin{question}
	Let $G=\langle \mathbf{x}\mid R\rangle$ be a residually finite Mel'nikov group with $|\mathbf{x}|\geqslant 3$. Does $G$ necessarily have negative immersions?
\end{question}

\subsection{The Surface Subgroup Conjecture for Hyperbolic Groups}
\label{sec:conjIF}

One of the most famous conjectures in geometric group theory, attributed to Gromov on Bestvina's list of questions \cite[Q 1.6]{Bestvina2004questions}, states that every one-ended hyperbolic group contains a surface subgroup.
This is known to be true in many cases (we give \cite{Wilton2018Essential} as a general reference), notably for the fundamental groups of hyperbolic $3$-manifolds \cite{Kahn2012Immersing}. We now prove that this conjecture implies \cref{conj:IF} thanks to a recent theorem of Mutanguha on hyperbolicity of ascending HNN extensions of free groups (it also follows from a recent preprint of Linton).
We also apply Kerckhoff's Nielsen Realisation Theorem:

\begin{theorem}[\cite{Kerckhoff1983}]
    \label{thm:virtual_surface}
    A torsion free group that is virtually a surface group is itself a surface group.
\end{theorem}
In the instances where we apply this theorem, the group surjects onto $\Z$ so an earlier result of Eckmann and M\"{u}ller is sufficient \cite[Theorem A]{EckmannMueller82}.

First we prove the following general result.

\begin{proposition}
\label{prop:hyperbolicity}
Let $G$ be a finitely generated group which admits an epimorphism to $\mathbb{Z}$. If all infinite-index subgroups of $G$ are free, then $G$ is hyperbolic or a surface group.
\end{proposition}

\begin{proof}
Suppose that $G$ is not hyperbolic. Then since the kernel of an epimorphism $G \to \Z$ is free, $G$ is the mapping torus of an automorphism of a (possibly infinite rank) free group and so it must contain a Baumslag--Solitar subgroup $H\cong \BS(1, n)$ with $|n|\geqslant 1$ by \cite[Corollary 5.3.6]{Mutanguha2021Dynamics}.
However, every infinite index subgroup of $G$ is free and so $H$ has finite index in $G$.
Now, $|n|=1$ as otherwise $H$ decomposes as $\mathbb{Z}[1/n]\rtimes\mathbb{Z}$ and so $H$, and hence $G$, contains a non-free subgroup of infinite index.
Thus $H$ is a surface group and so is $G$ by \cref{thm:virtual_surface}.
\end{proof}

We now conclude that \cref{conj:IF} follows from a weakened version of The Surface Subgroup Conjecture for Hyperbolic Groups, specific to one-relator groups.

\begin{theorem}
\label{thm:IF}
Suppose that every one-ended hyperbolic one-relator group contains a surface subgroup.
Let $G=\langle \mathbf{x}\mid R\rangle$ be an infinite non-free one-relator group with every subgroup of infinite index free. Then $G$ is a surface group.
\end{theorem}

\begin{proof}
By Proposition \ref{prop:hyperbolicity}, $G$ is a surface group, and we are done, or $G$ is hyperbolic.
(Linton has recently proven that every one-relator group whose presentation complex has negative immersions is hyperbolic \cite{linton2022one}, so this step also follows by combining \cref{corol:2GenIF} and \cref{thm:SGC_negative_immersions}, dealing with the cases of $|\mathbf{x}|=2$ and $|\mathbf{x}|\geqslant3$ respectively.)

Any splitting of the non-free hyperbolic group $G$ would imply the existence of a non-free infinite index subgroup, so $G$ is one-ended.
Thus in the case that it is hyperbolic, our assumptions tell us that it contains a surface subgroup $H$, which must have finite index in $G$.
As $G$ is torsion-free we are done by \cref{thm:virtual_surface}.
\end{proof}

That \cref{conj:IF} follows from The Surface Subgroup Conjecture for Hyperbolic Groups is now immediate.

\bibliographystyle{amsalpha}
\bibliography{BibTexBibliography}

\newcommand{\etalchar}[1]{$^{#1}$}
\providecommand{\bysame}{\leavevmode\hbox to3em{\hrulefill}\thinspace}
\providecommand{\MR}{\relax\ifhmode\unskip\space\fi MR }
\providecommand{\MRhref}[2]{%
  \href{http://www.ams.org/mathscinet-getitem?mr=#1}{#2}
}
\providecommand{\href}[2]{#2}
\begin{thebibliography}{FKM{\etalchar{+}}07}

\bibitem[Ago13]{Agol2013Haken}
Ian Agol, \emph{The virtual {H}aken conjecture}, Doc. Math. \textbf{18} (2013),
  1045--1087, With an appendix by Agol, Daniel Groves, and Jason Manning.
  \MR{3104553}

\bibitem[And73]{Andozhskii1973}
I.~V. Ando\v{z}ski\u{\i}, \emph{Demu\v{s}kin groups}, Akademiya Nauk SSSR.
  Matematicheskie Zametki \textbf{14} (1973), 121--126. \MR{338195}

\bibitem[Bau71]{Baumslag1971RF}
Gilbert Baumslag, \emph{Finitely generated cyclic extensions of free groups are
  residually finite}, Bull. Austral. Math. Soc. \textbf{5} (1971), 87--94.
  \MR{311776}

\bibitem[BDJ08]{Brownetal2008}
Nathanial~P. Brown, Kenneth~J. Dykema, and Kenley Jung, \emph{Free entropy
  dimension in amalgamated free products}, Proc. Lond. Math. Soc. (3)
  \textbf{97} (2008), no.~2, 339--367, With an appendix by Wolfgang L\"{u}ck.
  \MR{2439665}

\bibitem[Bes04]{Bestvina2004questions}
Mladen Bestvina, \emph{Questions in geometric group theory},
  \url{http://www.math.utah.edu/~bestvina/eprints/questions-updated.pdf}.

\bibitem[BFR19]{Baumslag2017survey}
Gilbert Baumslag, Benjamin Fine, and Gerhard Rosenberger, \emph{One-relator
  groups: an overview}, Groups {S}t {A}ndrews 2017 in {B}irmingham, London
  Math. Soc. Lecture Note Ser., vol. 455, Cambridge Univ. Press, Cambridge,
  2019, pp.~119--157. \MR{3931411}

\bibitem[BMS02]{Baumslag2002Open}
Gilbert Baumslag, Alexei~G. Myasnikov, and Vladimir Shpilrain, \emph{Open
  problems in combinatorial group theory. {S}econd edition}, Combinatorial and
  geometric group theory ({N}ew {Y}ork, 2000/{H}oboken, {NJ}, 2001), Contemp.
  Math., vol. 296, Amer. Math. Soc., Providence, RI, 2002, online version:
  \url{http://www.grouptheory.info}, pp.~1--38. \MR{1921705}

\bibitem[BNS87]{Bierietal1987}
Robert Bieri, Walter~D. Neumann, and Ralph Strebel, \emph{A geometric invariant
  of discrete groups}, Invent. Math. \textbf{90} (1987), no.~3, 451--477.
  \MR{914846 (89b:20108)}

\bibitem[Bri00]{Brinkmann2000}
P.~Brinkmann, \emph{Hyperbolic automorphisms of free groups}, Geom. Funct.
  Anal. \textbf{10} (2000), no.~5, 1071--1089. \MR{1800064}

\bibitem[CFR13]{Ciobanu2013surface}
L.~Ciobanu, B.~Fine, and G.~Rosenberger, \emph{The surface group conjecture:
  cyclically pinched and conjugacy pinched one-relator groups}, Results Math.
  \textbf{64} (2013), no.~1-2, 175--184. \MR{3095136}

\bibitem[Chi76]{Chiswell1976Euler}
Ian~M. Chiswell, \emph{Euler characteristics of groups}, Math. Z. \textbf{147}
  (1976), no.~1, 1--11. \MR{396785}

\bibitem[Coc54]{Cockcroft1954}
W.~H. Cockcroft, \emph{On two-dimensional aspherical complexes}, Proc. London
  Math. Soc. (3) \textbf{4} (1954), 375--384. \MR{0063042}

\bibitem[DL07]{DicksLinnell2007}
Warren Dicks and Peter~A. Linnell, \emph{{$L^2$}-{B}etti numbers of one-relator
  groups}, Math. Ann. \textbf{337} (2007), no.~4, 855--874. \MR{2285740}

\bibitem[EM82]{EckmannMueller82}
B.~Eckmann and H.~M\"{u}ller, \emph{Plane motion goups and virtual
  {P}oincar\'{e} duality of dimension two}, Invent. Math. \textbf{69} (1982),
  no.~2, 293--310. \MR{674408}

\bibitem[FH99]{FeighnHandel1999}
Mark Feighn and Michael Handel, \emph{Mapping tori of free group automorphisms
  are coherent}, Ann. of Math. (2) \textbf{149} (1999), no.~3, 1061--1077.
  \MR{1709311}

\bibitem[FKM{\etalchar{+}}07]{Fine2007surface}
B.~Fine, O.~G. Kharlampovich, A.~G. Myasnikov, V.~N. Remeslennikov, and
  G.~Rosenberger, \emph{On the surface group conjecture}, Sci. Ser. A Math.
  Sci. (N.S.) \textbf{15} (2007), 1--15. \MR{2367908}

\bibitem[HW15]{HagenWise2015}
Mark~F. Hagen and Daniel~T. Wise, \emph{Cubulating hyperbolic free-by-cyclic
  groups: the general case}, Geom. Funct. Anal. \textbf{25} (2015), no.~1,
  134--179. \MR{3320891}

\bibitem[Ker83]{Kerckhoff1983}
Steven~P. Kerckhoff, \emph{The {N}ielsen realization problem}, Ann. of Math.
  (2) \textbf{117} (1983), no.~2, 235--265. \MR{690845 (85e:32029)}

\bibitem[Kie20]{Kielak2020}
Dawid Kielak, \emph{The {B}ieri--{N}eumann--{S}trebel invariants via {N}ewton
  polytopes}, Invent. Math. \textbf{219} (2020), no.~3, 1009--1068.
  \MR{4055183}

\bibitem[KM12]{Kahn2012Immersing}
Jeremy Kahn and Vladimir Markovic, \emph{Immersing almost geodesic surfaces in
  a closed hyperbolic three manifold}, Ann. of Math. (2) \textbf{175} (2012),
  no.~3, 1127--1190. \MR{2912704}

\bibitem[KM21]{KhukhroMazurow2014}
E.~I. Khukhro and V.~D. Mazurov, \emph{Unsolved problems in group theory. the
  {Kourovka} notebook}, arXiv:1401.0300v22 (2021).

\bibitem[Lin93]{Linnell1993}
Peter~A. Linnell, \emph{Division rings and group von {N}eumann algebras}, Forum
  Math. \textbf{5} (1993), no.~6, 561--576. \MR{1242889}

\bibitem[Lin22]{linton2022one}
Marco Linton, \emph{One-relator hierarchies}, arXiv:2202.11324 (2022).

\bibitem[LM89]{LubotzkyMann1989rank}
Alexander Lubotzky and Avinoam Mann, \emph{Residually finite groups of finite
  rank}, Math. Proc. Cambridge Philos. Soc. \textbf{106} (1989), no.~3,
  385--388. \MR{1010362}

\bibitem[LS01]{LyndonSchupp2001}
Roger~C. Lyndon and Paul~E. Schupp, \emph{Combinatorial group theory}, Classics
  in Mathematics, Springer-Verlag, Berlin, 2001, Reprint of the 1977 edition.
  \MR{1812024}

\bibitem[LW21]{Louder2021uniform}
Larsen Louder and Henry Wilton, \emph{Uniform negative immersions and the
  coherence of one-relator groups}, arXiv:2107.08911 (2021).

\bibitem[LW22]{Louder2018negative}
\bysame, \emph{Negative immersions for one-relator groups}, Duke Math. J.
  \textbf{171} (2022), no.~3, 547--594. \MR{4382976}

\bibitem[Lyn50]{Lyndon1950}
Roger~C. Lyndon, \emph{Cohomology theory of groups with a single defining
  relation}, Ann. of Math. (2) \textbf{52} (1950), 650--665. \MR{47046}

\bibitem[MKS04]{mks}
Wilhelm Magnus, Abraham Karrass, and Donald Solitar, \emph{Combinatorial group
  theory}, second ed., Dover Publications, Inc., Mineola, NY, 2004,
  Presentations of groups in terms of generators and relations. \MR{2109550}

\bibitem[MS90]{MannSegal1990uniform}
Avinoam Mann and Dan Segal, \emph{Uniform finiteness conditions in residually
  finite groups}, Proc. London Math. Soc. (3) \textbf{61} (1990), no.~3,
  529--545. \MR{1069514}

\bibitem[Mut21]{Mutanguha2021Dynamics}
Jean~Pierre Mutanguha, \emph{The dynamics and geometry of free group
  endomorphisms}, Adv. Math. \textbf{384} (2021), Paper No. 107714, 60.
  \MR{4237417}

\bibitem[Ol'79]{Olshanskii79noetherian}
A.~Yu. Ol'shanskii, \emph{An infinite simple torsion-free {N}oetherian group},
  Izv. Akad. Nauk SSSR Ser. Mat. \textbf{43} (1979), no.~6, 1328--1393.
  \MR{567039}

\bibitem[Sta65]{Stallings1965}
John Stallings, \emph{Homology and central series of groups}, J. Algebra
  \textbf{2} (1965), 170--181. \MR{175956}

\bibitem[Str77]{Strebel1977}
R.~Strebel, \emph{A remark on subgroups of infinite index in {P}oincar\'{e}
  duality groups}, Comment. Math. Helv. \textbf{52} (1977), no.~3, 317--324.
  \MR{457588}

\bibitem[Wil12]{Wilton2012}
Henry Wilton, \emph{One-ended subgroups of graphs of free groups with cyclic
  edge groups}, Geom. Topol. \textbf{16} (2012), no.~2, 665--683. \MR{2928980}

\bibitem[Wil18]{Wilton2018Essential}
\bysame, \emph{Essential surfaces in graph pairs}, J. Amer. Math. Soc.
  \textbf{31} (2018), no.~4, 893--919. \MR{3836561}

\bibitem[Wis22]{Wise2020coherence}
Daniel~T. Wise, \emph{Coherence, local indicability and nonpositive
  immersions}, J. Inst. Math. Jussieu \textbf{21} (2022), no.~2, 659--674.
  \MR{4386825}

\end{thebibliography}
\end{document}